\documentclass[12pt, a4paper]{article}
\usepackage{amsmath}
\usepackage{comment}
\usepackage{amsfonts}
\usepackage{amssymb}
\usepackage{epsfig}
\usepackage{graphicx}
\usepackage{color}
\usepackage{amsthm}
\usepackage{enumerate}
\usepackage [latin1]{inputenc}
\usepackage[numbers, sort&compress]{natbib}
\usepackage{url}

\newtheorem{thm}{Theorem}[section]
\newtheorem{lem}[thm]{Lemma}

\theoremstyle{definition}

\graphicspath{{figures/}}

\usepackage{url}

\usepackage{authblk}

\long\def\delete#1{}

\usepackage{xcolor}
\usepackage[normalem]{ulem}

\begin{document}
\openup 0.5\jot

\title{The maximum forcing numbers of quadriculated tori}
\author[1]{Qianqian Liu\thanks{ E-mail: \texttt{liuqq2023@imut.edu.cn.}}}
\author[2]{Yaxian Zhang\thanks{E-mail: \texttt{yxzhang2016@lzu.edu.cn.}}}
\author[2]{Heping Zhang\footnote{The corresponding author. E-mail: \texttt{zhanghp@lzu.edu.cn.}}}

\affil[1]{\small College of Science, Inner Mongolia University of Technology, Hohhot, Inner Mongolia 010010, China}
\affil[2]{\small School of Mathematics and Statistics, Lanzhou University, Lanzhou, Gansu 730000, China}

\date{}

\maketitle

\setlength{\baselineskip}{20pt}
\noindent {\bf Abstract}:
Klein and Randi\'{c} (1985) proposed the concept of forcing number, which has an application in chemical resonance theory.
Let $G$ be a graph with a perfect matching $M$. The forcing number of $M$ is the smallest cardinality of a subset of $M$ that is contained only in one perfect matching $M$. The maximum forcing number of $G$ is the maximum value of forcing numbers over all perfect matchings of $G$. Kleinerman (2006) obtained that the maximum forcing number of $2n\times 2m$ quadriculated torus is $nm$. By improving Kleinerman's approach, we obtain the maximum forcing numbers of all 4-regular quadriculated graphs on torus except one class.

\vspace{2mm} \noindent{\textbf{Keywords}}  Perfect matching, maximum forcing number, quadriculated torus
\vspace{2mm}


\section{\normalsize Introduction}
Let $G$ be a graph with a perfect matching $M$. A subset $S\subseteq M$ is called a \emph{forcing set} of $M$ if it is contained in no other perfect matchings of $G$. The smallest cardinality of a forcing set of $M$ is called the \emph{forcing number} of $M$, denoted by $f(G,M)$. The \emph{minimum} and \emph{maximum forcing number} of $G$, denoted by $f(G)$ and  $F(G)$, are respectively defined as the minimum and maximum values of $f(G,M)$ over all perfect matchings $M$ of $G$.

The concept of the forcing number of a perfect matching was first introduced by Klein and Randi\'{c} \cite{3,klein85} in 1985 when they studied the molecular resonance structures, which was called ``innate degree of freedom'' in chemical literatures. It was turned out that the perfect matchings with the maximum forcing number contribute more to the stability of molecule\cite{32}.
Afshani, Hatami and Mahmoodian \cite{5} pointed out that the computational complexity of the maximum forcing number of a graph is still an open problem.
Xu, Bian and Zhang \cite{27} obtained that maximum forcing numbers of hexagonal systems are equal to the resonant numbers. The same result also holds for polyominoes \cite{zhou2016,lin2017} and BN-fullerene graphs \cite{40}. Abeledo and Atkinson \cite{13} had already obtained that resonant numbers of 2-connected plane bipartite graphs can be computed in polynomial time. Thus, the maximum forcing numbers of such three classes of graphs can be solved in polynomial time.

The cartesian product of graphs $G$ and $H$ is denoted by $G\square H$. The maximum forcing numbers of the cartesian product of some special graphs, such as paths and cycles, have been obtained. Let $P_n$ and $C_n$ denote a path and a cycle with $n$ vertices, respectively. Pachter and Kim \cite{6}, Lam and Pachter \cite{9} obtained that $F(P_{2n}\square P_{2n})=n^2$ using different methods. In general, Afshani et al. \cite{5} proved that
$F(P_m\square P_n)=\lfloor\frac{m}{2}\rfloor\cdot\lfloor\frac{n}{2}\rfloor$ for even $mn$. Besides, they \cite{5} obtained that $F(P_{2m}\square C_{2n})=mn$ and $F(P_{2m+1}\square C_{2n})=mn+1$, and asked such a question: what is the maximum forcing number of a non-bipartite cylinder $P_{2m}\square C_{2n+1}$? Jiang and Zhang \cite{29} solved this problem and obtained that $F(P_{2m}\square C_{2n+1})=m(n+1)$. By a method of marking independent sets, Kleinerman \cite{16} obtained that $F(C_{2m}\square C_{2n})=mn$. Obviously, $C_{2m}\square C_{2n}$ is a special type of 4-regular quadriculated graphs on torus.

As early as 1991, Thomassen \cite{Tho} classified all 4-regular quadriculated graphs on torus (abbreviated to ``\emph{quadriculated tori}'') into two classes, which were reduced into one class by Li \cite{classfy}.
For $n\geq1$ and $m\geq 2$, a \emph{quadriculated torus} $T(n,m,r)$ is obtained from an $n\times m$ chessboard ($n$ rows, each consists of $m$ squares) by sticking the left and right sides together and then identifying the top and bottom sides with a torsion of $r$ squares where $1\leq r\leq m$ (see Fig. \ref{torsion}). Obviously, $T(n,m,m)$ is isomorphic to $C_n\square C_m$. Based on the parity of three parameters, quadriculated tori with perfect matchings can be divided into six classes $T(2n,2m,2r)$, $T(2n,2m,2r-1)$, $T(2n+1,2m,2r)$, $T(2n+1,2m,2r-1)$, $T(2n,2m+1,2r)$ and $T(2n,2m+1,2r-1)$.
\begin{figure}[h]
\centering
\includegraphics[height=3cm,width=6cm]{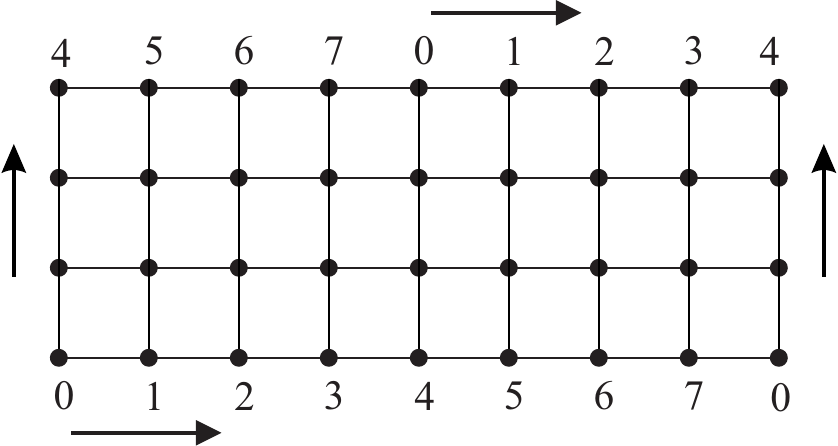}
\caption{\label{torsion}Quadriculated torus $T(3,8,4)$.}
\end{figure}

In this paper, we obtain a simple expression for the maximum forcing numbers of all quadriculated tori except for $T(2n+1,2m,2r-1)$. In Section 2, we give some notations and terminologies, and prove some crucial lemmas. In Section 3, we prove that $F(T(2n,2m+1,t))=n(m+1)$ for $1\leq t\leq 2m+1$ by choosing a fixed independent set. In Section 4, we obtain that $F(T(2n,2m,r))=mn+1$ if $(r,2m)=2$, and $F(T(2n,2m,r))=mn$ otherwise, where $(r,2m)$ represents the greatest common factor of $r$ and $2m$, and $1\leq r\leq 2m$. In Section 5, by another representation of the quadriculated torus, we obtain the maximum forcing number of $T(2n+1,2m,2r)$ for $1\leq r\leq m$.

\section{\normalsize Preliminaries}%
In this section, we give some notations and terminologies, and prove some important lemmas.

Let $T(n,m,r)$ be a quadriculated tori. According to positions of vertices in the chessboard, we label the vertices of $T(n,m,r)$ as $\{v_{i,j}| i\in Z_n, j \in Z_m\}$ (see Fig. \ref{nota}), where $Z_m:=\{0,1,\dots,m-1\}$. Hence $v_{i,0}$ is adjacent to $v_{i,m-1}$ for $i\in Z_{n}$, and $v_{0,j}$ is adjacent to $v_{n-1,m-r+j}$ for $j\in Z_{m}$.
\begin{figure}[h]
\centering
\includegraphics[height=3.3cm,width=7cm]{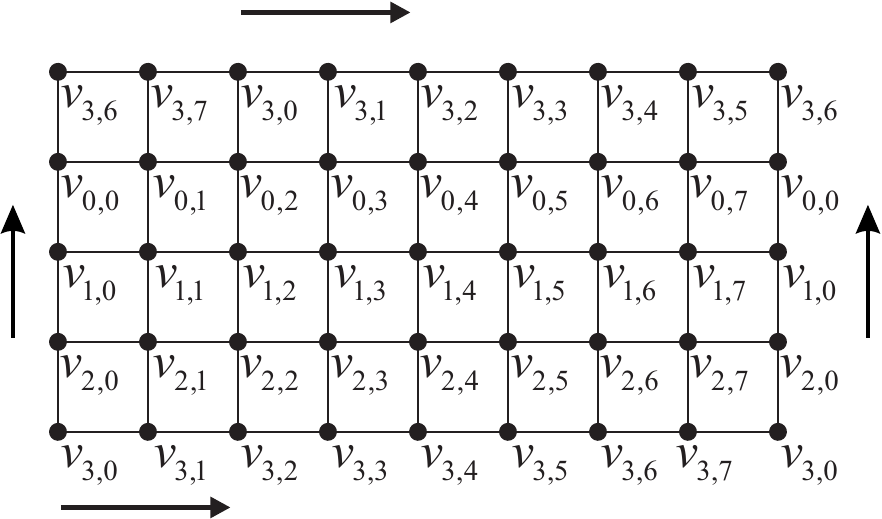}
\caption{\label{nota}Labels of the vertices in $T(4,8,2)$.}
\end{figure}

For $j\in Z_m$, let $v_{0,j}v_{1,j}\cdots v_{n-1,j}$ be a path called \emph{$j$-column},
and $v_{0,j}$ and $v_{n-1,j}$ are \emph{initial} and \emph{terminal} of $j$-column. For convenience, we call $j$-column a \emph{column} for $j\in Z_{m}$.
If initial $v_{0,j_2}$ of $j_2$-column is adjacent to terminal $v_{n-1,j_1}$ of $j_1$-column, that is, $j_2\equiv j_1+r$ (mod $m$), then $j_2$-column is the \emph{successor} of $j_1$-column.
Let $j_0$-, $j_1$-, \dots, $j_{g-1}$-columns be pairwise different such that $j_{k+1}$-column is the successor of $j_k$-column for each $k\in Z_g$. Then these $g$ columns form a cycle, called an \emph{$\mathrm{I}$-cycle}.  In \cite{LYZ}, we had proved the following lemma.
\begin{lem}\rm{\cite{LYZ}}\label{lem1} $T(n,m,r)$ has $(r,m)$ $\mathrm{I}$-cycles and each $\mathrm{I}$-cycle contains $\frac{m}{(r,m)}$ columns. Moreover, any consecutive $(r,m)$ columns lie on different $\mathrm{I}$-cycles.
\end{lem}

Intuitively, we call $v_{i,j}v_{i,j+1}$ a \emph{horizontal edge} and $v_{i,j}v_{i+1,j}$ a \emph{vertical edge} for $i\in Z_n$ and $j\in Z_{m}$.
Obviously, all vertical edges form $(r,m)$ $\mathrm{I}$-cycles, and all horizontal edges form $n$ $\mathrm{II}$-cycles (consisting of all vertices and edges on a row). Preserving the horizontal and vertical edges, we can obtain another representation of this quadriculated tori, denoted by $T^*(n,m,r)$, in which all vertices of a $\mathrm{I}$-cycle of $T(n,m,r)$ lie on a column and all vertices of a $\mathrm{II}$-cycle of $T(n,m,r)$ are divided into different rows (see Fig. \ref{obsev}). Therefore, $\mathrm{I}$-cycles (resp. $\mathrm{II}$-cycles) in $T(n,m,r)$ corresponds to $\mathrm{II}$-cycles (resp. $\mathrm{I}$-cycles) in $T^*(n,m,r)$.
For $i\in Z_{n}$, the subgraph of $T(n,m,r)$ induced by all vertices of any consecutive two rows $$\{v_{i,0},v_{i,1},\dots, v_{i,m-1}\}\cup \{v_{i+1,0},v_{i+1,1},\dots, v_{i+1,m-1}\}$$ is denoted by $R_{i,i+1}$. Then $R_{i,i+1}$ contains a subgraph isomorphic to $C_{m}\square P_2$. Particularly, $R_{i,i+1}$ is isomorphic to $C_{m}\square P_2$ for $n\geq 2$ where $i\in Z_n$.

Relabeling the vertices of $T(n,m,r)$ according to $\mathrm{I}$-cycle, we can obtain the following lemma. For details, see Section 2 of ref. \cite{LYZ}.
\begin{figure}[h]
\centering
\includegraphics[height=5.7cm,width=13cm]{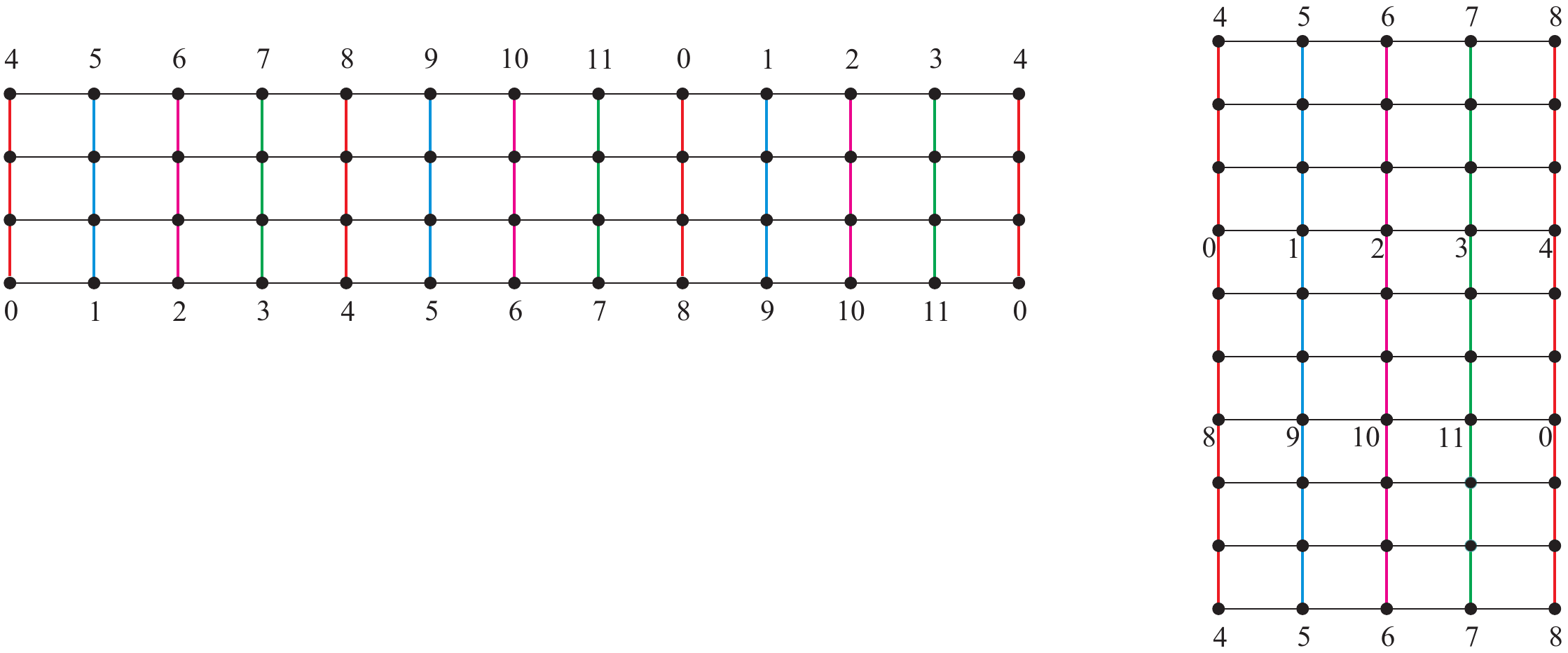}
\caption{\label{obsev} Quadriculated tori $T(3,12,8)$ and $T(4,9,3)=T^*(3,12,8)$.}
\end{figure}
\begin{lem}\rm{\cite{LYZ}}\label{drawing} For $n\geq1$, $m\geq 2$ and $1\leq r\leq m$, $T^*(n,m,r)=T((r,m), \frac{mn}{(r,m)},(\frac{m}{(r,m)}-k)n)$, where $0\leq k\leq \frac{m}{(r,m)}-1$ is an integer satisfying the equation $(r,m)\equiv rk\ (\text{mod\ }  m).$ Furthermore, $T^{**}(n,m,r)=T(n,m,r)$.
\end{lem}

For a non-empty subset $S\subseteq V(G)$, the \emph{subgraph induced by $S$}, denoted by $G[S]$, is a graph whose vertex set is $S$ and edge set consists of those edges of $G$ that have both end vertices in $S$. The induced subgraph $G[V(G)\setminus S]$ is denoted by $G-S$. For an edge subset $F\subseteq E(G)$, we use $V(F)$ to denote the set of all end vertices of edges in $F$.

Let $G$ be a graph with a perfect matching $M$. We give an independent set $T$ of $G$ called \emph{marked vertices} of $G$. Define
$M_T=\{e\in M\ |\ e \text{\ has an end vertex in }T\}.$ Then $M_T\subseteq  M$ and $|M_T|=|T|$. A cycle of $G$ is \emph{$M$-alternating} if its edges appear alternately in $M$ and off $M$.
\begin{lem}\label{forcingset} Let $G$ be a graph with a perfect matching $M$. If the union of all paths of length 2 whose initial and terminal lie in $T$ contains no $M$-alternating cycles, then $f(G,M)\leq |M|-|T|$.
\end{lem}
\begin{proof}We prove that $G[V(M_T)]$ contains no $M$-alternating cycles. Suppose to the contrary that $G[V(M_T)]$ contains an $M$-alternating cycle $C$. Then $C$ is also an $M_T$-alternating cycle. Since $T$ is an independent set, half vertices of $C$ are marked, and marked and unmarked vertices appear alternately. Thus, $C$ can be viewed as the union of paths of length two whose initial and terminal lie in $T$, which is a contradiction.

Since $G[V(M_T)]$ contains no $M$-alternating cycles, $G[V(M_T)]$ has a unique perfect matching. Thus, $M\setminus M_T$ is a forcing set of $M$, and $f(G,M)\leq |M\setminus M_T|=|M|-|T|$.
\end{proof}

For convenience, ``the union of all paths of length 2 whose initial and terminal are marked vertices'' is defined as ``\emph{marked subgraph}''.

Next we give the concept of $2\times 2$-polyomino, which is a kind of general ``marked subgraph''.
A \emph{polyomino} is a finite connected subgraph in the infinite plane square grid in which every interior face is surrounded by a square and every edge belongs to at least one square. A \emph{$2\times 2$-polyomino} is also a polyomino which is obtained by replacing each square in a polyomino by a $2\times 2$ chessboard (see Fig. \ref{polyominog}).
\begin{figure}[h]
\centering
\includegraphics[height=3.2cm,width=7cm]{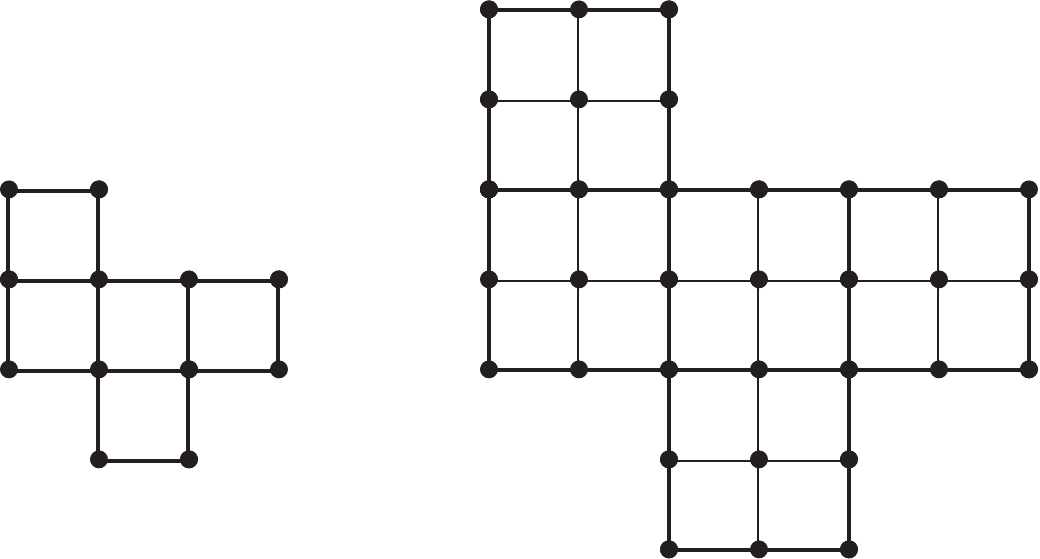}
\caption{\label{polyominog} A polyomino and its corresponding $2\times 2$-polyomino.}
\end{figure}

An \emph{interior vertex} of a plane graph is a vertex which is not on the boundary of the unbounded face. For a polyomino, an interior vertex means a vertex of degree 4. By the proof of Theorem 3.2 in \cite{29}, Jiang and Zhang obtained the following result.
\begin{lem}\label{polyomino}\rm{\cite{29}} A $2\times 2$-polyomino has an odd number of interior vertices.
\end{lem}

\section{\normalsize  The maximum forcing number of $T(2n,2m+1,r)$ for $1\leq r\leq 2m+1$}
In this section, we will obtain the maximum forcing number of $T(2n,2m+1,r)$ by the method of marking independent sets for $1\leq r\leq 2m+1$.

For $T(2n,m,r)$, we define some subsets of vertices and edges. For $i\in Z_{n}$, let $$X_{i}=\{v_{i,2k}|k\in Z_{\lfloor\frac{m}{2}\rfloor}\}  \text{   and   } Y_{i}=\{v_{i,2k+1}|k\in Z_{\lfloor\frac{m}{2}\rfloor}\}.$$ For $j\in Z_{m}$, let $W_{j}=\{v_{2k,j}v_{2k+1,j}|k\in Z_{n}\}$, $$W^{1}_{j}=\{v_{4k+2,j}v_{4k+3,j}|k\in Z_{\lfloor\frac{n}{2}\rfloor}\} \text{  and  }  W^{2}_{j}=\{v_{4k,j}v_{4k+1,j}|k\in Z_{\lfloor\frac{n+1}{2}\rfloor}\}$$ be two subsets of $W_j$.
\begin{thm}\label{odd} For $n, m\geq 1$ and $1\leq r\leq 2m+1$, $F(T(2n,2m+1,r))=(m+1)n$.
\end{thm}
\begin{proof}
Let $M_1=W_0\cup W_1\cup \cdots \cup W_{2m}$ be a perfect matching of $T(2n,2m+1,r)$ (see Fig. \ref{fig111}). We will prove that $f(T(2n,2m+1,r),M_1)=(m+1)n$.
\begin{figure}[h]
\centering
\includegraphics[height=3.6cm,width=11.8cm]{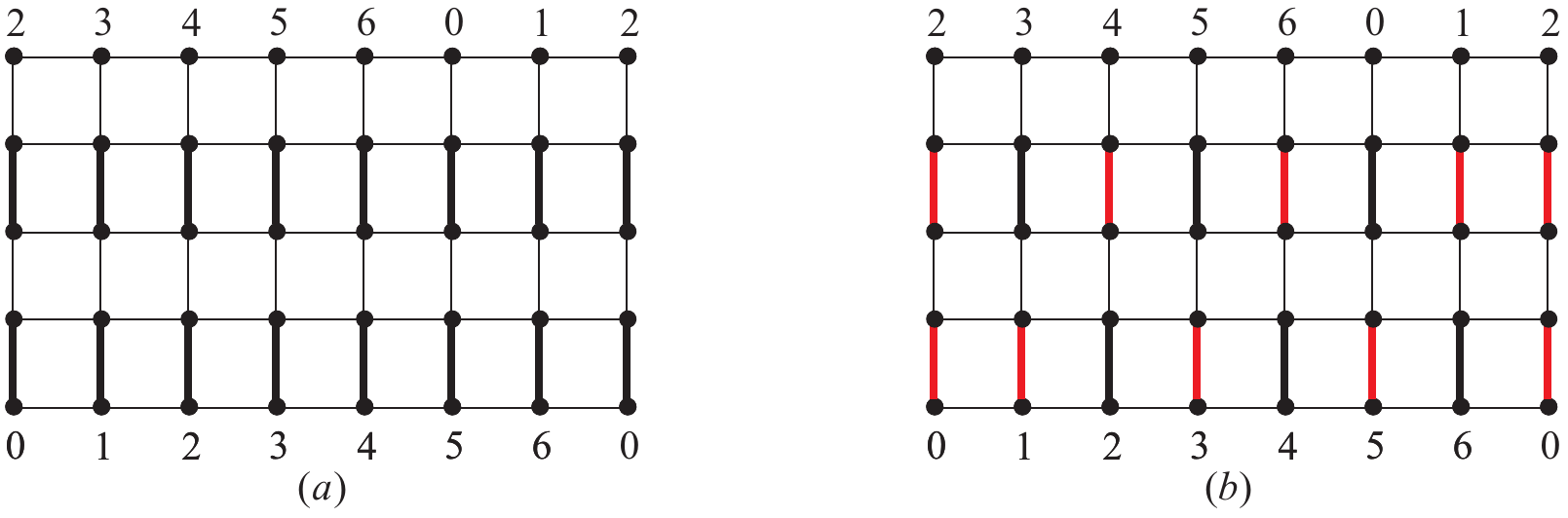}
\caption{\label{fig111}The perfect matching $M_1$ of $T(4,7,5)$, and a forcing set of $M_1$ shown in red lines.}
\end{figure}
For $i\in Z_n$, since $R_{2i,2i+1}$ contains a subgraph isomorphic to $C_{2m+1}\square P_2$, any forcing set of $M_1\cap E(R_{2i,2i+1})$ has size at least $m+1$. Thus, $M_1$ has the forcing number at least $n(m+1)$. Let $S=W_0\cup W^1_1\cup W^2_2\cup W^1_3\cup W^2_4\cup \cdots \cup W^1_{2m-1}\cup  W^2_{2m}$ be a subset of $M_1$ shown as red lines in Fig. \ref{fig111}(b), so that exactly $m+1$ edges of $R_{2i,2i+1}$ are chosen to belong to $S$. Obviously, $S$ is a forcing set of $M_1$ with size $n(m+1)$. Hence, we obtain that  $f(T(2n,2m+1,r), M_1)=n(m+1)$.

For any perfect matching $M$ of $T(2n,2m+1,r)$, we will choose an independent set $T$ of size $mn$ such that ``marked subgraph'' contains no $M$-alternating cycles. By Lemma \ref{forcingset}, we have $$f(T(2n,2m+1,r),M)\leq |M|-|T|=(2m+1)n-mn=(m+1)n.$$
By the arbitrariness of $M$, we have $F(T(2n,2m+1,r))\leq(m+1)n$.

To achieve this goal, we will take $m$ appropriate vertices on 1, 3, $\dots$, $2n-1$ rows. Let $X'_{i}=(X_i-\{v_{i,0}\})\cup \{v_{i,2m}\}$ for $i\in Z_{2n-1}$ and $$X^*=\{v_{2n-1,2m+1-r}\}\cup\{v_{2n-1,2m+1-r+j}|j=3,5,\dots,2m-1\}.$$
Take marked vertices $T=X'_1\cup X'_3\cup \cdots \cup X'_{2n-3}\cup X^*$ shown as Fig. \ref{fig112}.
\begin{figure}[h]
\centering
\includegraphics[height=4.8cm,width=16cm]{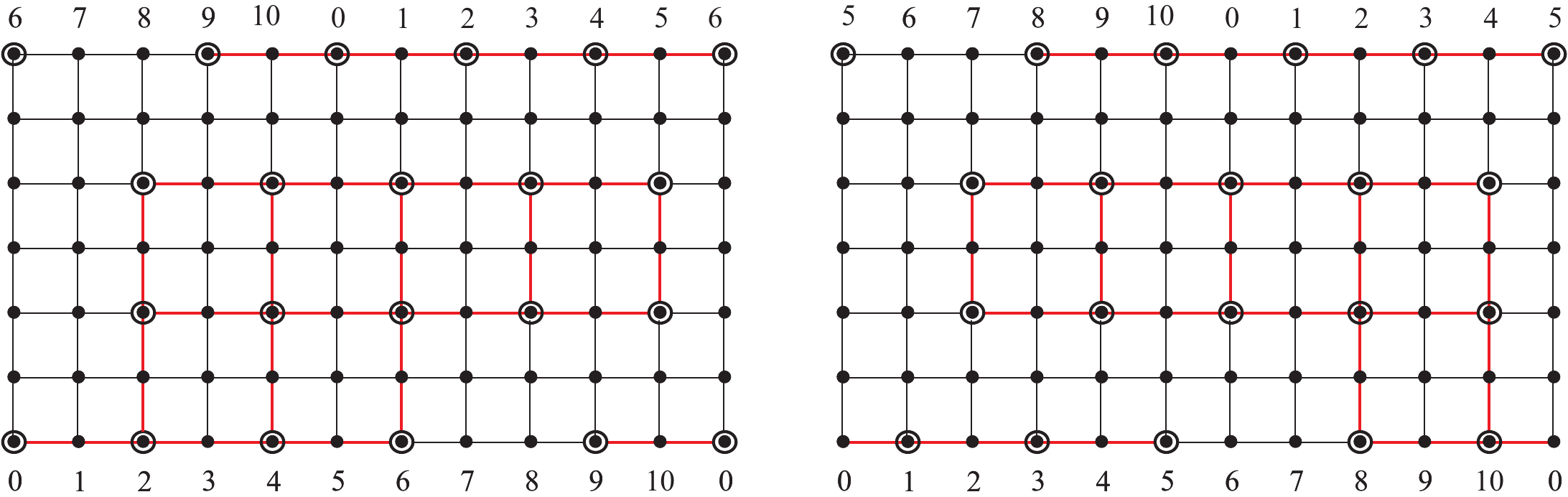}
\caption{\label{fig112}Marked vertices of $T(6,11,5)$  and $T(6,11,6)$.}
\end{figure}

From left to right, we choose 1'st, 4'th, 6'th, $\dots$, $(2m)$'th vertices in the first row and 3'th, 5'th, $\dots$, $(2m+1)$'th vertices in the third row as marked vertices. Hence, all edges incident with $v_{0,j}$ are not contained in ``marked subgraph'' for $0\leq j\leq 2m$. Thus such $2m+1$ vertices are not contained in ``marked subgraph'', and ``marked subgraph'' is a plane graph. The ``marked subgraph'' formed by all paths of length two whose initial and terminal are in $X'_{1}\cup X'_{3}\cup \cdots \cup X'_{2n-3}$ is a $2\times 2$-polyomino corresponding to a $(n-2)\times (m-1)$ chessboard, and the ``marked subgraph'' formed by all paths of length two whose initial and terminal are in $X'_{2n-3}\cup X^*$ is a $2\times 2$-polyomino corresponding to some $1\times t$ $(0\leq t\leq m-1)$ chessboard attaching a path. Thus, ``marked subgraph'' is a $2\times 2$-polyomino attaching a path.

Suppose to the contrary that $C$ is an $M$-alternating cycle contained in ``marked subgraph''. Then $\text{Int}[C]$ (the subgraph of $T(2n,2m+1,r)$ induced by the vertices of $C$ and its interior) is a $2\times 2$-polyomino. By Lemma \ref{polyomino}, $\text{Int}[C]$ has an odd number of interior vertices, which contradicts that $C$ is $M$-alternating. Thus, ``marked subgraph'' contains no $M$-alternating cycles.
\end{proof}

\section{\normalsize  The maximum forcing number of $T(2n,2m,r)$  for $1\leq r\leq 2m$}
In this section, we are to obtain the maximum forcing number of $T(2n,2m,r)$ for $1\leq r\leq 2m$.

In the proof of Theorem \ref{odd}, we fix $mn$ marked vertices to prove that ``marked subgraph'' contains no $M$-alternating cycles for any perfect matching $M$ of  $T(2n,2m+1,r)$, where $1\leq r\leq 2m+1$. But for a perfect matching $M$ of $T(2n,2m,r)$, ``marked subgraph'' contains an $M$-alternating cycle no matter which sets with size $mn$ we mark. For the case that each $\mathrm{II}$-cycle is not $M$-alternating, we can prove the following result.
\begin{lem}\label{modifiedcycle}For $n,m\geq 2$ and $1\leq r\leq 2m$, assume that $M$ is a perfect matching of $T(2n,2m,r)$ and each $\mathrm{II}$-cycle is not $M$-alternating.  Then we can  mark  $mn$ vertices so that ``marked subgraph'' contains no $M$-alternating cycles.
\end{lem}
\begin{proof} First we choose an independent set $T$ of $T(2n,2m,r)$ with size $mn$ as marked vertices. If $n$ is odd, then take $$T=\{Y_{4k+1}|k=0,1,2, \dots, \frac{n-1}{2}\} \bigcup \{X_{4k+3}|k=0,1,2, \dots, \frac{n-3}{2}\}.$$ Otherwise, take $$T=\{Y_{4k+1}|k=0,1,2, \dots, \frac{n-2}{2}\} \bigcup \{X_{4k+3}|k=0,1,2, \dots, \frac{n-2}{2}\}.$$ See two examples in Fig. \ref{em81}.
\begin{figure}[h]
\centering
\includegraphics[height=6cm,width=13cm]{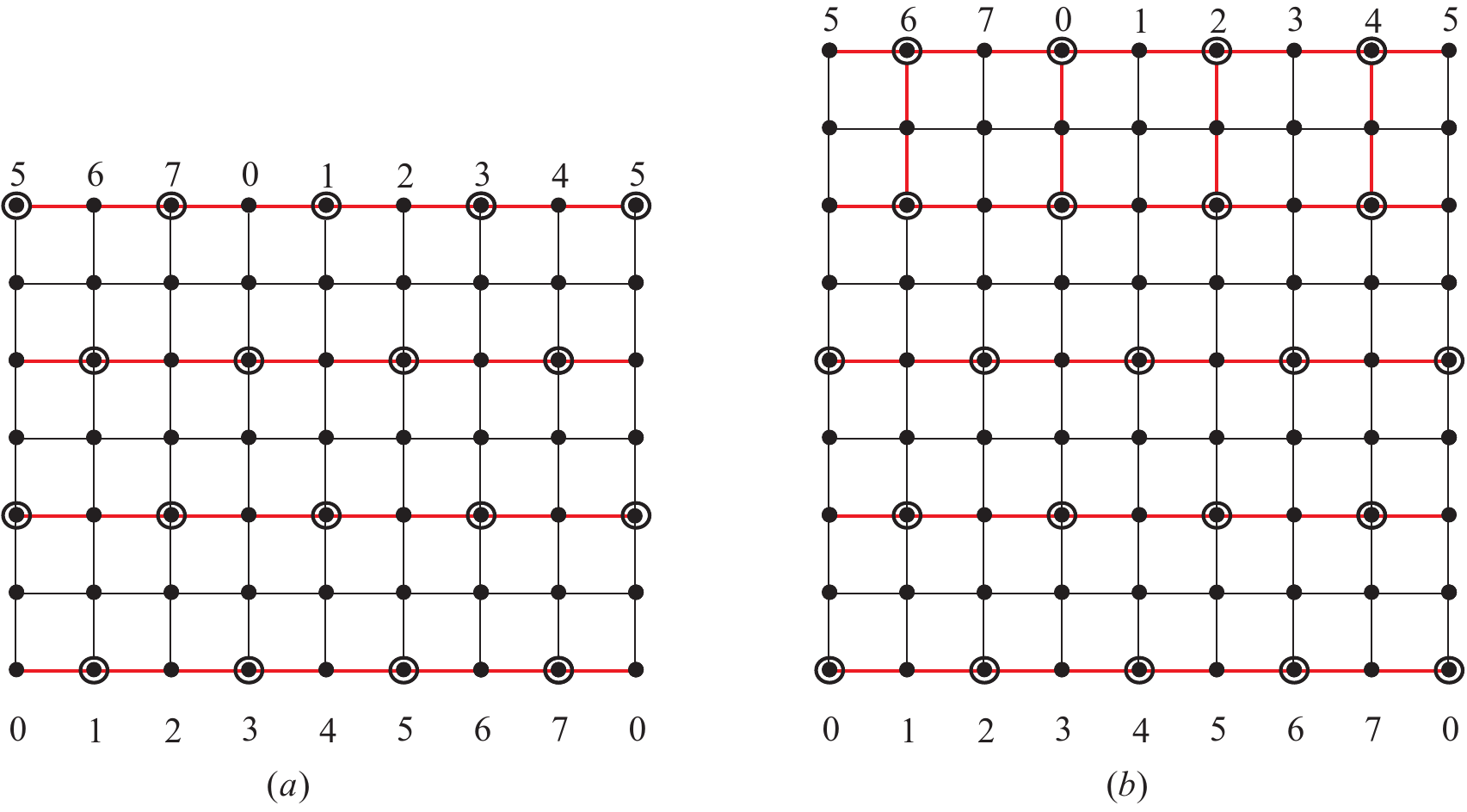}
\caption{\label{em81}Marked vertices and ``marked subgraph'' of $T(6,8,3)$ and $T(8,8,3)$.}
\end{figure}
If $r$ is odd (resp. even), then marked vertices on the first and last rows are located at different (resp. same) columns. For the case that $r$ and $n$ have the same parity, ``marked subgraph'' consists of $n$ $\mathrm{II}$-cycles. By the assumption, each $\mathrm{II}$-cycle is not $M$-alternating. Thus, ``marked subgraph'' contains no $M$-alternating cycles, and $T$ is the marked vertices we require. It suffices to consider the case that $r$ and $n$ have different parity.

In the sequel, we only prove the lemma for the case that $r$ is odd and $n$ is even, and the proof is similar for the other case. Now marked vertices on the first and third rows are located at the same columns. Thus ``marked subgraph'' consists of $m$ paths of length two $\{v_{2n-1,2m-r+j}v_{0,j}v_{1,j}|j=1,3,\dots,2m-1\}$ and $n$ $\mathrm{II}$-cycles shown as red lines in Fig. \ref{em81}(b).

By the assumption, each $\mathrm{II}$-cycle is not $M$-alternating. Hence, each $M$-alternating cycle (if exists) of ``marked subgraph'' is contained in the subgraph induced by all vertices of the first three rows, and contains at least two vertices on the second row. By Lemma \ref{polyomino}, an $M$-alternating cycle cannot form the boundary of a $2\times 2$-polyomino which corresponds to a $1\times l$ chessboard for $1\leq l\leq m-1$. Therefore, any $M$-alternating cycle of ``marked subgraph'' has the following form: it starts with a $\mathrm{II}$-cycle in the first row and moves to the third row and backs at specified intervals shown as green lines in Fig. \ref{emmm}(a).
Notice that each such cycle contains exactly $2m$ horizontal edges, divided in some way between the two rows.
\begin{figure}[h]
\centering
\includegraphics[height=2.6cm,width=17cm]{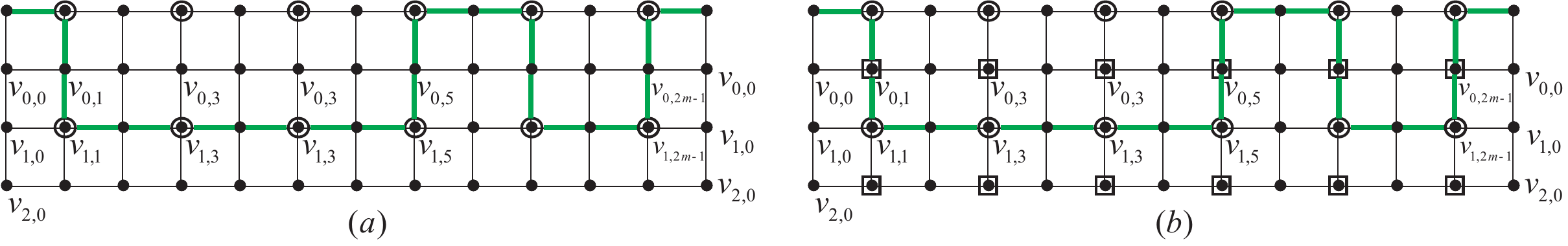}
\caption{\label{emmm}$M$-alternating cycle of ``marked subgraph''.}
\end{figure}

Translating the marked vertices down by one row shown as Fig. \ref{emmm}(b), we also have an $M$-alternating cycle lying on the subgraph induced by the vertices of the second, third and fourth rows (otherwise, new marked vertices we obtained is what we want). We will demonstrate that the new $M$-alternating cycle has more horizontal edges in the bottom (i.e., the fourth) row than the first one does. Consider the set of horizontal edges in the bottom row of the first $M$-alternating cycle, which is partitioned into subsets naturally by proximity: there is a set of horizontal edges, then a cross-over, then perhaps a cross-back, then another set of horizontal edges, and so forth. Consider one of these sets, say  $\{v_{1,1}v_{1,2},v_{1,2}v_{1,3},\cdots, v_{1,2t}v_{1,2t+1}\}$ shown as green lines on the third row of Fig. \ref{emm8}(a), where $t\geq 1$. By the form of $M$-alternating cycles, edges of $\{v_{1,1}v_{0,1},v_{0,1}v_{2n-1,2m-r+1}\}$ and $\{v_{1,2t+1}v_{0,2t+1},v_{0,2t+1}v_{2n-1,2m-r+2t+1}\}$ are contained in the first $M$-alternating cycle. It suffices to prove that the set of edges $$\{v_{2,0}v_{2,1},v_{2,1}v_{2,2},v_{2,2}v_{2,3},\cdots, v_{2,2t}v_{2,2t+1}\} \text{ or } \{v_{2,1}v_{2,2},v_{2,2}v_{2,3},\cdots, v_{2,2t}v_{2,2t+1},v_{2,2t+1}v_{2,2t+2}\}$$ is contained in the bottom row of the new $M$-alternating cycle.
 \begin{figure}[h]
\centering
\includegraphics[height=2.6cm,width=17cm]{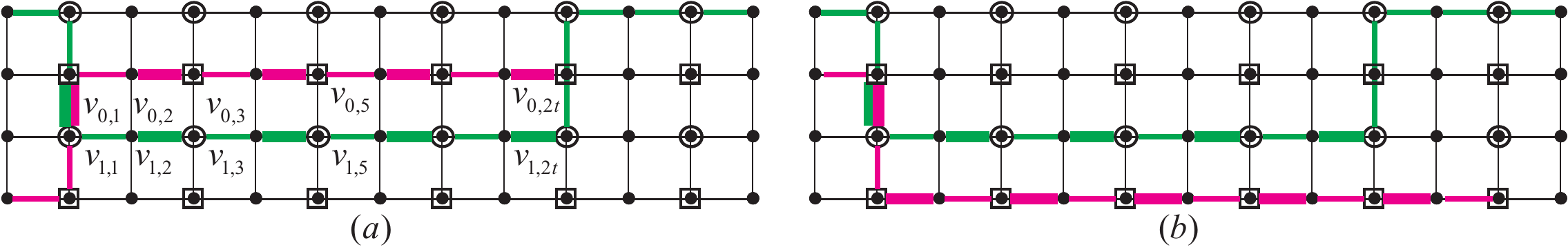}
\caption{\label{emm8}Part of the two $M$-alternating cycles lying in corresponding ``marked subgraphs''.}
\end{figure}

Since all horizontal edges of the first $M$-alternating cycle lie on the first and third rows, and these of the new $M$-alternating cycle lie on the second and fourth rows, only  vertical edges in $\{v_{0,2k+1}v_{1,2k+1}|k=0,1,\dots, m-1\}$ may be intersected. If $v_{0,1}v_{1,1}$ belongs to the new $M$-alternating cycle, then $v_{0,1}v_{1,1}\in M$, and $v_{1,1}v_{2,1}$ is contained in the new $M$-alternating cycle.  We claim that $v_{0,0}v_{0,1}$ is contained in the new $M$-alternating cycle. Otherwise, $v_{0,1}v_{0,2}$ and $v_{0,2}v_{0,3}\in M$ are contained in the new $M$-alternating cycle. Since $v_{1,2}v_{1,3}\in M$, $v_{0,3}v_{1,3}$ does not lie on the new $M$-alternating cycle. Hence the path $v_{0,1}v_{0,2}v_{0,3}\cdots v_{0,2t}v_{0,2t+1}$ lies on the new $M$-alternating cycle (see Fig. \ref{emm8}(a)). Note that $v_{0,2t}v_{0,2t+1}\in M$, which contradicts that $v_{2n-1,2m-r+2t+1}v_{0,2t+1}$ and $v_{0,2t+1}v_{1,2t+1}$ belong to the first $M$-alternating cycle. Now we prove the claim. Thus, $v_{0,0}v_{0,1}$ and $v_{1,1}v_{2,1}$ lie on the new $M$-alternating cycle (see Fig. \ref{emm8}(b)). Since $v_{1,1}v_{1,2}v_{1,3}\cdots v_{1,2t}v_{1,2t+1}$ is on the first $M$-alternating cycle, we can obtain that the path $v_{2,1}v_{2,2}v_{2,3}\cdots v_{2,2t}v_{2,2t+1}v_{2,2t+2}$ lies on the second $M$-alternating cycle by a simple argument. If $v_{0,2t+1}v_{1,2t+1}$ belongs to the new $M$-alternating cycle, then, by a similar argument, we can obtain that $$v_{0,2t+2}v_{0,2t+1}v_{1,2t+1}v_{2,2t+1}v_{2,2t}\cdots v_{2,2}v_{2,1}v_{2,0}$$ lies on the second $M$-alternating cycle. If neither $v_{0,1}v_{1,1}$ nor $v_{0,2t+1}v_{1,2t+1}$ belongs to the new $M$-alternating cycle (see Fig. \ref{emm82222}), then, by the form of $M$-alternating cycles, such two $M$-alternating cycles have no common edges in this area, and the result holds naturally.
This means that all horizontal edges in the bottom row of the first $M$-alternating cycle give rise to abutting horizontal edges in the bottom row of the second one. Because the intersected vertical edges cannot overlap, there is at least one more horizontal edge in the bottom row of the second $M$-alternating cycle.
 \begin{figure}[h]
\centering
\includegraphics[height=2cm,width=8cm]{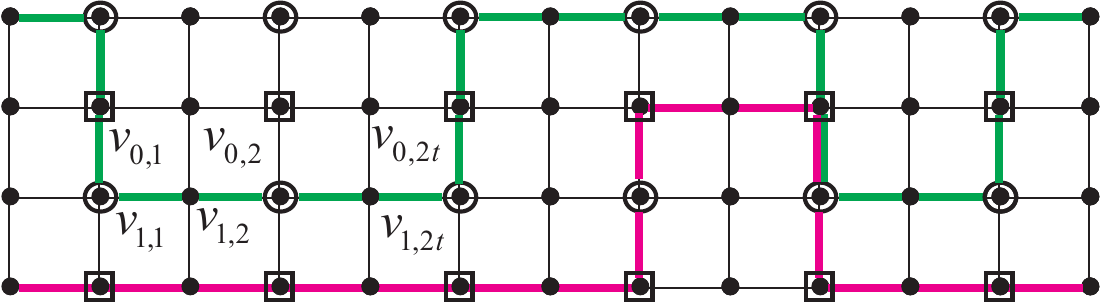}
\caption{\label{emm82222}Part of the two $M$-alternating cycles lying in corresponding ``marked subgraphs''.}
\end{figure}

Each time we translate the marked vertices down by one row, we obtain an abutting $M$-alternating cycle which contains more  horizontal edges in the bottom row than the first one does. Since any $M$-alternating cycle contains no more than $2m$ horizontal edges on its bottom row, there is a placement of marked vertices such that ``marked subgraph'' contains no $M$-alternating cycles.
\end{proof}

\subsection{\small The maximum forcing number of $T(2n,2m,2r)$ for $1\leq r\leq m$}

By Lemma \ref{lem1}, $T(n,m,r)$ contains $(r,m)$ $\mathrm{I}$-cycles, and each $\mathrm{I}$-cycle contains $\frac{mn}{(r,m)}$ vertices. For $(r,m)\geq 2$ and $j\in Z_{(r,m)}$, the subgraph induced by all vertices of the two $\mathrm{I}$-cycles containing $j$-column and $(j+1)$-column contains a subgraph isomorphic to $C_{\frac{mn}{(r,m)}}\square P_2$, denoted by $C_{j,j+1}$. Particularly, $C_{j,j+1}$ is isomorphic to $C_{\frac{mn}{(r,m)}}\square P_2$ for $(r,m)\geq 3$ where $j\in Z_{(r,m)}$.

\begin{thm}\label{mqps1}For $n,m\geq 2$ and $1\leq r\leq m$, we have
\begin{equation*}
 F(T(2n,2m,2r))=
 \begin{cases}
 mn+1, & \quad {if\  (r,m)=1};\\
 mn,&\quad {otherwise}.
 \end{cases}
 \end{equation*}
\end{thm}
\begin{proof}First we prove the case that $(r,m)\neq 1$. Let $M_1=E_0\cup E_2\cup \dots \cup E_{2m-2}$ be a perfect matching of $T(2n,2m,2r)$ shown as Fig. \ref{em1}(a), where $E_j=\{v_{i,j}v_{i,j+1}|i\in Z_{2n}\}$. Then $C_{2j,2j+1}$ contains a subgraph isomorphic to $C_{\frac{2mn}{(r,m)}}\square P_2$ for $j\in Z_{(r,m)}$ and contains $\frac{mn}{(r,m)}$ disjoint $M_1$-alternating cycles. Hence, $T(2n,2m,2r)$ contains $mn$ disjoint $M_1$-alternating cycles and $f(T(2n,2m,2r),M_1)\geq mn$. Form a forcing set of size $mn$ so that half horizontal edges of $C_{2j,2j+1}$ are chosen for $j\in Z_{(r,m)}$. Precisely, from top to bottom we choose 1'th, 3'th, $\dots$, $(\frac{2mn}{(r,m)}-1)'$th horizontal edges of $C_{4j,4j+1}$ for $j\in \lceil\frac{(r,m)}{2}\rceil$ and 2'th, 4'th, $\dots$, $\frac{2mn}{(r,m)}$'th horizontal edges of $C_{4j+2,4j+3}$ for $j\in \lfloor\frac{(r,m)}{2}\rfloor$ (red lines of $T^*(2n,2m,2r)$ in Fig. \ref{em1}(b) and that of $T(2n,2m,2r)$ in Fig. \ref{em1}(c) form a forcing set). Hence, $f(T(2n,2m,2r),M_1)= mn$. 
\begin{figure}[h]
\centering
\includegraphics[height=5.5cm,width=14cm]{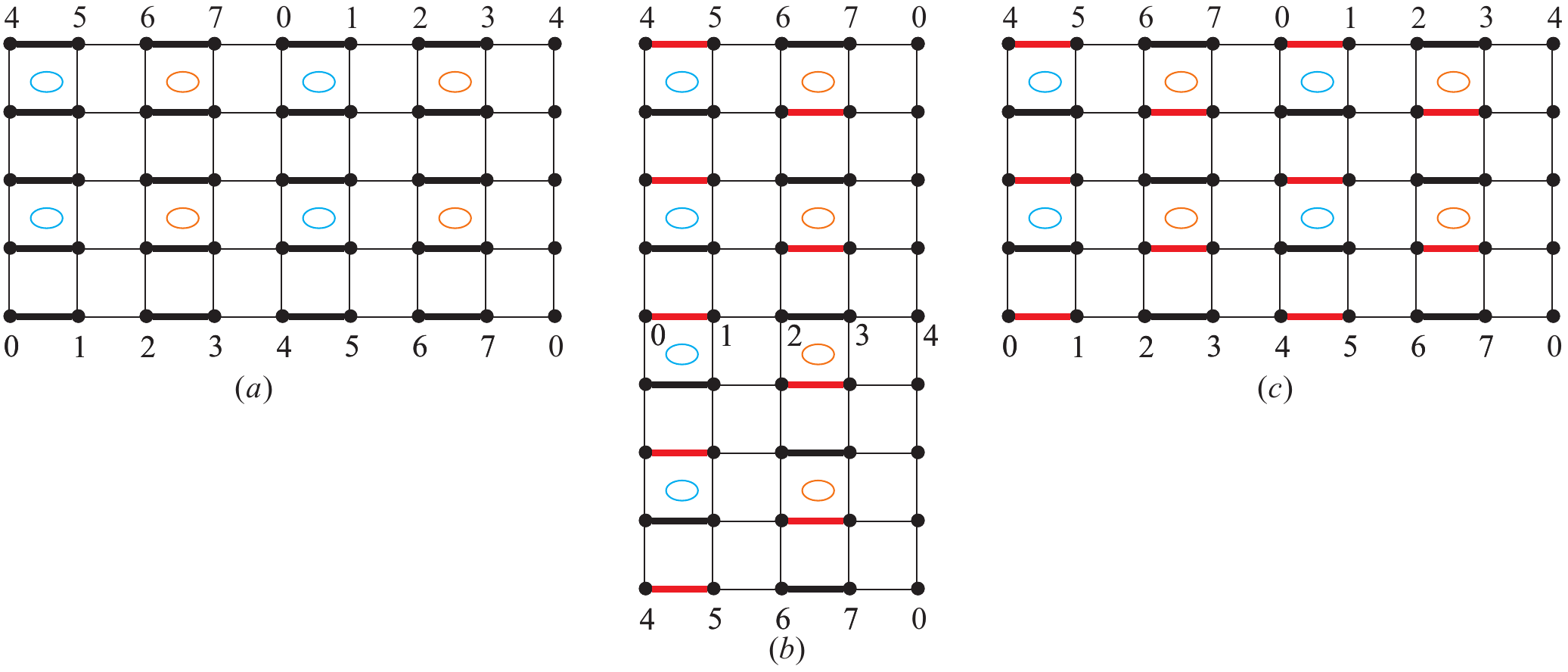}
\caption{\label{em1}The perfect matching $M_1$ of $T(4,8,4)$, where red lines form a forcing set of $M_1$.}
\end{figure}

Let $M$ be any perfect matching of $T(2n,2m,2r)$. It suffices to prove that $$f(T(2n,2m,2r),M)\leq mn.$$
If none of $\mathrm{II}$-cycles is $M$-alternating, then we can mark $mn$ vertices so that ``marked subgraph'' contains no $M$-alternating cycles by Lemma \ref{modifiedcycle}. Otherwise, there is an $M$-alternating $\mathrm{II}$-cycle. Then each $\mathrm{I}$-cycle is not $M$-alternating.
By Lemma \ref{drawing}, $T(2n,2m,2r)$ has another representation $$T^*(2n,2m,2r)=T(2(r,m), \frac{2nm}{(r,m)},2n(\frac{m}{(r,m)}-k)),$$ in which each $\mathrm{II}$-cycle is not $M$-alternating.
By Lemma \ref{modifiedcycle}, we can mark $mn$ vertices so that ``marked subgraph'' contains no $M$-alternating cycles.
By Lemma \ref{forcingset}, $$f(T(2n,2m,2r),M)=f(T^*(2n,2m,2r),M)\leq |M|-|T|=mn.$$ By the arbitrariness of $M$, we have $F(T(2n,2m,2r))\leq mn$.

Next we prove the case that $(r,m)= 1$. By Lemma \ref{lem1}, $T(2n,2m,2r)$ has exactly two $\mathrm{I}$-cycles. Let $M_1=E_0\cup E_2\cup \dots \cup E_{2m-2}$ be a perfect matching of $T(2n,2m,2r)$ shown as bold lines in Fig. \ref{em12}(a).
\begin{figure}[h]
\centering
\includegraphics[height=3.5cm,width=14cm]{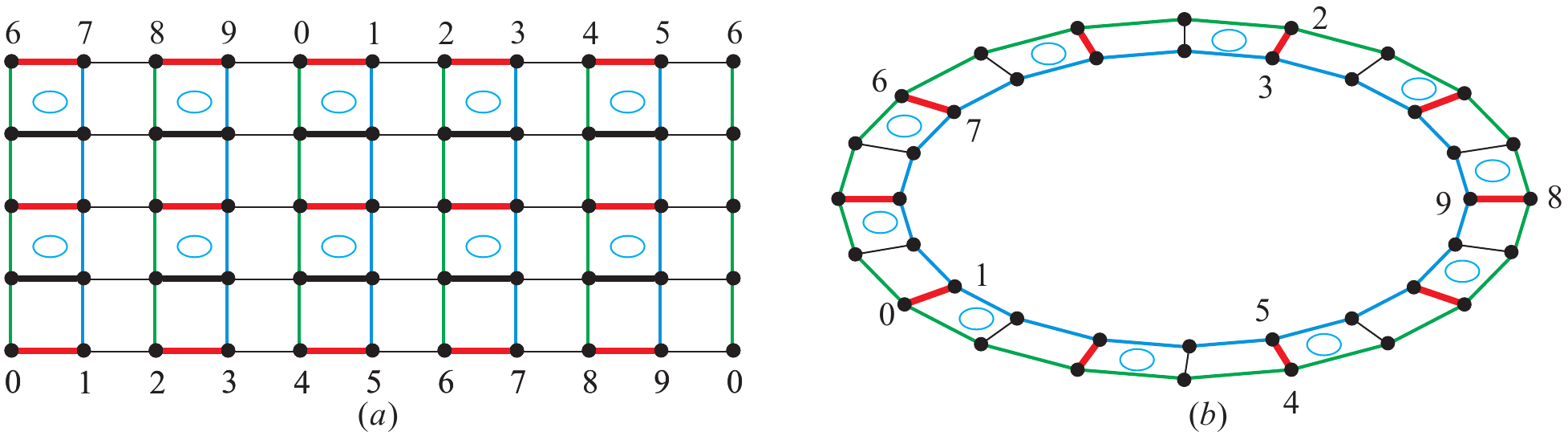}
\caption{\label{em12}The perfect matching $M_1$ of $T(4,10,4)$, and red lines cannot form a forcing set of $M_1$.}
\end{figure}
Since $C_{0,1}$ contains a subgraph isomorphic to $C_{2nm}\square P_2$, $T(2n,2m,2r)$ contains $mn$ disjoint $M_1$-alternating cycles. Since a forcing set of $M_1$ contains at least one edge from each $M_1$-alternating cycle, any forcing set of $M_1$ has size at least $mn$. To find a forcing set of size $mn$, we need to choose one of the horizontal edges in any two consecutive ones of $C_{0,1}$. In $C_{0,1}$, starting with the two consecutive edges $v_{0,0}v_{0,1}$ and $v_{1,0}v_{1,1}$, in which the latter are chosen, we choose a set of horizontal edges with size $mn$ shown as red lines in Fig. \ref{em12}(b), where each $E_{2j}$ for $j\in Z_{m}$ has $n$ edges $\{v_{2i+1,2j}v_{2i+1,2j+1}|i\in Z_n\}$ being chosen.
But the chosen $mn$ edges cannot form a forcing set of $M_1$ for there are still $n$ $\mathrm{II}$-cycles being not intersected with such $mn$ edges (see red lines in Fig. \ref{em12}(a)). Hence, $f(T(2n,2m,2r),M_1)\geq mn+1$. It's easy to find a forcing set of size $mn+1$. Thus $f(T(2n,2m,2r),M_1)=mn+1$.

For any perfect matching $M$ of $T(2n,2m,2r)$, we are to prove that $$f(T(2n,2m,2r),M)\leq mn+1.$$ By Lemma \ref{forcingset}, it suffices to prove that we can mark at least $mn-1$ vertices in $T(2n,2m,2r)$ such that ``marked subgraph'' contains no $M$-alternating cycles. If each $\mathrm{II}$-cycle is not $M$-alternating, then we can mark $mn$ vertices so that ``marked subgraph'' contains no $M$-alternating cycles by Lemma \ref{modifiedcycle}. Otherwise, assume that $v_{2n-1,0}v_{2n-1,1}\cdots v_{2n-1,2m-1}v_{2n-1,0}$ is an $M$-alternating cycle, and $\{v_{2n-1,2j}v_{2n-1,2j+1}|j\in Z_{m}\}\subseteq M$. Let $$X_*=\{v_{0,1},v_{0,3},\dots,v_{0,2r-1},v_{0,2r+3},v_{0,2r+5},\dots,v_{0,2m-1}\} \text{ and } Y_*=\{v_{3,0},v_{5,0},\dots,v_{2n-1,0}\}.$$ Take $T=Y_*\cup X_*\cup X'_2\cup X'_4\cup \dots \cup X'_{2n-2}$ as marked vertices shown as Fig. \ref{em122}, where $X'_{i}=X_{i}-\{v_{i,0}\}$ for $i\in Z_{2n}$. Then all vertices on the third row don't lie on the ``marked subgraph'', and ``marked subgraph'' is a plane graph shown as red lines in Fig. \ref{em122}.
\begin{figure}[h]
\centering
\includegraphics[height=5.5cm,width=12.5cm]{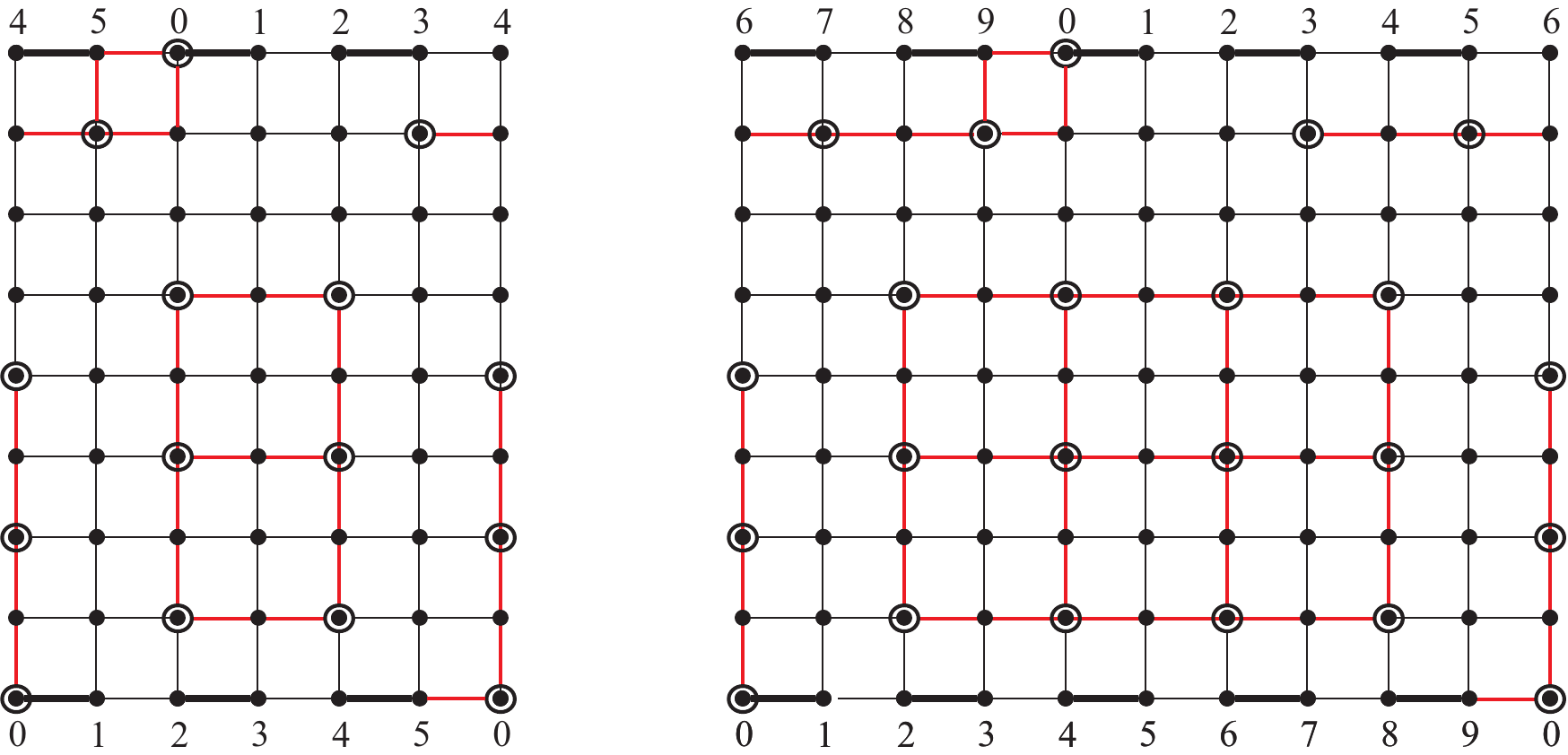}
\caption{\label{em122}Marked vertices and ``marked subgraph'' of $T(8,6,2)$ and $T(8,10,4)$.}
\end{figure}

The ``marked subgraph'' formed by all paths of length two whose initial and terminal are in $X'_2\cup X'_4 \cup \cdots \cup X'_{2n-2}$ is a $2\times 2$-polyomino corresponding to
a $(n-2)\times (m-2)$ chessboard. Noting that both $v_{2n-1,0}$ and $v_{0,2r-1}$ are marked vertices, $v_{2n-1,0}v_{2n-1,2m-1}v_{0,2r-1}v_{0,2r}v_{2n-1,0}$ is contained in ``marked subgraph'', and the ``marked subgraph'' formed by all paths of length two whose initial and terminal are in $X_*\cup Y_*$ is a cycle of length 4 attaching a path on $2m-2$ vertices and a path on $2n-3$ vertices. Furthermore, ``marked subgraph'' consists of a $2\times 2$-polyomino corresponding to a $(n-2)\times (m-2)$ chessboard and a 4-cycle attaching a path on $2m-2$ vertices and a path on $2n-3$ vertices.
Since $v_{2n-1,0}v_{2n-1,1}\in M$, such 4-cycle $v_{2n-1,0}v_{2n-1,2m-1}v_{0,2r-1}v_{0,2r}v_{2n-1,0}$ is not $M$-alternating. By Lemma \ref{polyomino}, a $2\times 2$-polyomino contains no $M$-alternating cycles. Thus, ``marked subgraph'' contains no $M$-alternating cycles.

By Lemma \ref{forcingset},  $M\setminus E_{T}$ is a forcing set of $M$ and $$f(T(2n,2m,2r),M)\leq |M|-|T|\leq 2mn-(mn-1)=mn+1.$$ By the arbitrariness of $M$, we have $F(T(2n,2m,2r))\leq nm+1$.
\end{proof}

\subsection{\small The maximum forcing number of $T(2n,2m,2r-1)$ for $1\leq r\leq m$}

Next we will obtain the maximum forcing number of $T(2n,2m,2r-1)$ for $1\leq r\leq m$.

\begin{thm}\label{even}For $n\geq1$, $m\geq 2$ and $1\leq r\leq m$, $F(T(2n,2m,2r-1))=mn$.
\end{thm}
\begin{proof}
Let $M_1=W_0\cup W_1\cup \cdots \cup W_{2m-1}$ be a perfect matching of $T(2n,2m,2r-1)$. Since $R_{2i,2i+1}$ contains a subgraph isomorphic to $C_{2m}\square P_2$, it contains $m$ disjoint $M_1$-alternating cycles for $i\in Z_n$. Thus, any forcing set of $M_1$ has size at least $mn$. Clearly, $W^2_0\cup W^1_1\cup W^2_2\cup \cdots \cup W^2_{2m-2}\cup W^1_{2m-1}$ shown as red lines in Fig. \ref{fig11} is a forcing set of $M_1$ with size $mn$. Hence, we obtain that $f(T(2n,2m,2r-1), M_1)=mn$.
\begin{figure}[h]
\centering
\includegraphics[height=4.2cm,width=15cm]{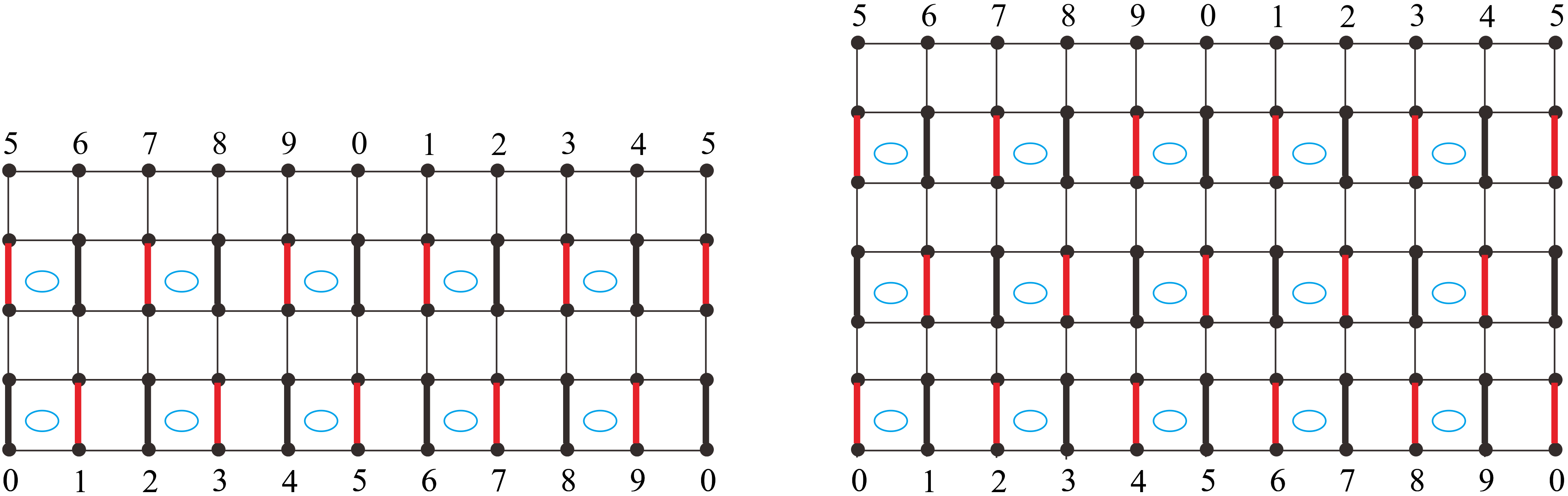}
\caption{\label{fig11}Perfect matchings $M_1$ of $T(4,10,5)$ and $T(6,10,5)$, where red lines form a forcing set.}
\end{figure}

Let $M$ be any perfect matching of $T(2n,2m,2r-1)$, we are to prove that $$f(T(2n,2m,2r-1),M)\leq mn.$$ It suffices to mark $mn$ vertices of $T(2n,2m,2r-1)$ such that ``marked subgraph'' contains no $M$-alternating cycles. If we have done, then by Lemma \ref{forcingset}, we have $$f(T(2n,2m,2r-1),M)\leq |M|-mn=mn.$$
By the arbitrariness of $M$, we have $F(T(2n,2m,2r-1))\leq mn$.

For $n\geq 2$, we only suffice to prove the case that there is a $\mathrm{II}$-cycle is $M$-alternating by Lemma \ref{modifiedcycle}. For $n=1$, $n$ and $2r-1$ are of the same parity, by the proof of Lemma \ref{modifiedcycle}, we also need to prove the same case as $n\geq 2$. Without loss of generality, we suppose that $v_{2n-1,0}v_{2n-1,1}\cdots v_{2n-1,2m-1}v_{2n-1,0}$ is an $M$-alternating $\mathrm{II}$-cycle, and $\{v_{2n-1,2j}v_{2n-1,2j+1}|j\in Z_m\}\subseteq M.$  Let $T=Y_*\cup X'_0 \cup X'_2\cup \cdots \cup X'_{2n-2}$ (see Fig. \ref{mmark2}) as marked vertices, where $$Y_*=\{v_{2n-1,2m-2r+1},v_{1,0}, v_{3,0},\dots, v_{2n-3,0}\} \text{ and } X'_{i}=X_{i}-\{v_{i,0}\} \text{ for } i\in Z_{2n}.$$ Then $T$ is of size $mn$. Since any  vertices of $Y_*$ and that of $X'_{2i}$ belong to no same rows for $i\in Z_{n}$, any vertices of $\{v_{i,1}, v_{i,2m-1}|i\in Z_{2n}\}$ are not contained in ``marked subgraph''. Furthermore, any vertices of $\{v_{2n-1,2m-2r+1+j}|j=2,3,\dots,2m-2\}$ are not contained in ``marked subgraph''. Thus, ``marked subgraph'' is a plane graph shown as red lines in Fig. \ref{mmark2}.
The ``marked subgraph'' formed by all paths of length two whose initial and terminal are in $X'_0\cup X'_2\cup X'_4 \cup \cdots \cup X'_{2n-2}$ is a $2\times 2$-polyomino corresponding to
a $(n-1)\times (m-2)$ chessboard, which contains no $M$-alternating cycles by Lemma \ref{polyomino}.
\begin{figure}[h]
\centering
\includegraphics[height=4.6cm,width=13.5cm]{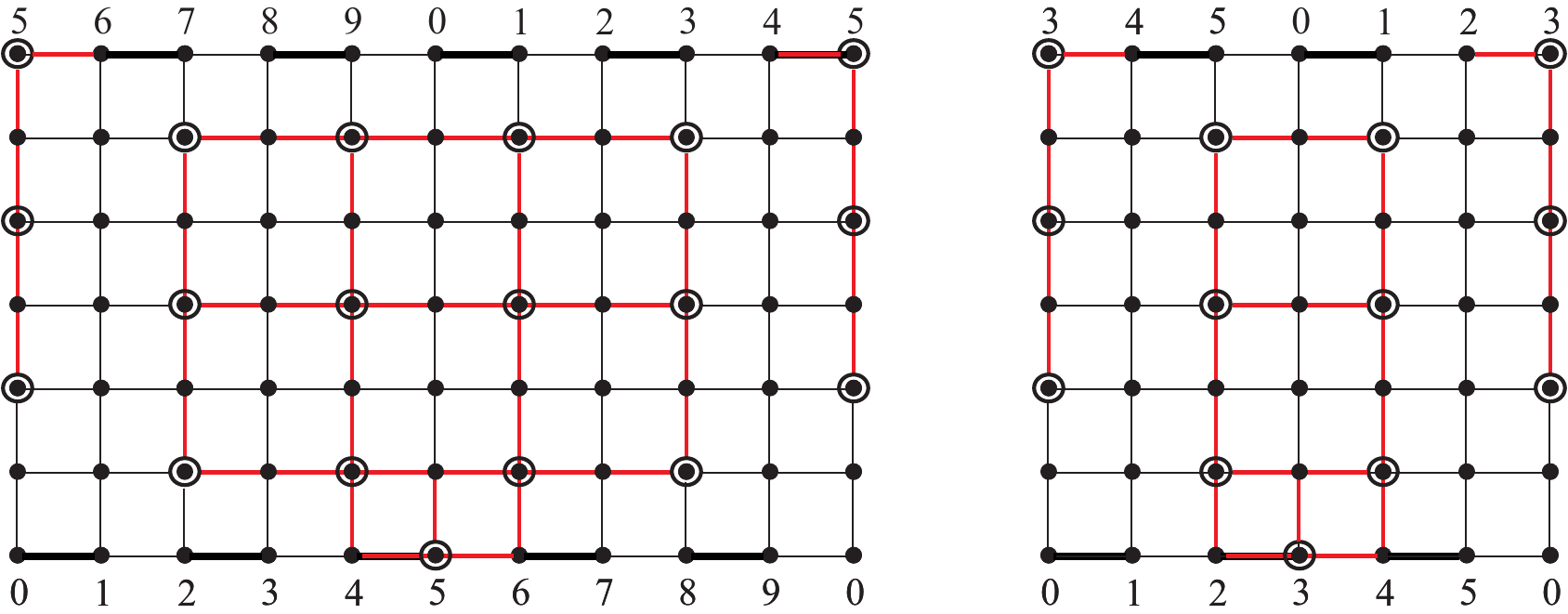}
\caption{\label{mmark2}Marked vertices and ``marked subgraph'' of $T(6,10,5)$ and $T(6,6,3)$.}
\end{figure}

Since $v_{2n-1,2m-2r+1}$, $v_{2n-2,2m-2r}$ and $v_{2n-2,2m-2r+2}$ are marked vertices, four paths of length two $v_{2n-2,2m-2r}v_{2n-1,2m-2r}v_{2n-1,2m-2r+1}$, $v_{2n-2,2m-2r}v_{2n-2,2m-2r+1}v_{2n-1,2m-2r+1}$, $v_{2n-2,2m-2r+1}\\v_{2n-2,2m-2r+2}v_{2n-1,2m-2r+2}$ and $v_{2n-2,2m-2r+1}v_{2n-1,2m-2r+1}v_{2n-1,2m-2r+2}$ are contained in ``marked subgraph''. Let $C$ be an $M$-alternating cycle of ``marked subgraph''. Then $C$ contains the vertex $v_{2n-1,2m-2r+1}$. Since $C$ is $M$-alternating, it also contains three edges $v_{2n-1,2m-2r}v_{2n-2,2m-2r}$, $v_{2n-1,2m-2r}v_{2n-1,2m-2r+1}$ and $v_{2n-1,2m-2r+1}v_{2n-2,2m-2r+1}$, and such four vertices $v_{2n-1,2m-2r}$,\\ $v_{2n-1,2m-2r+1}$, $v_{2n-2,2m-2r}$ and $v_{2n-2,2m-2r+1}$ are on the boundary of $\text{Int}[C]$. Next, we prove that $C$ contains exactly such four vertices. If $C$ contains at least six vertices, then $\text{Int}[C]$ and $\text{Int}[C]-\{v_{2n-1,2m-2r}, v_{2n-1,2m-2r+1}\}$ have the same number of interior vertices. Since $\text{Int}[C]-\{v_{2n-1,2m-2r}, v_{2n-1,2m-2r+1}\}$ is a $2\times 2$-polyomino, it has an odd number of interior vertices by Lemma \ref{polyomino}. Thus, $\text{Int}[C]$ has an odd number of interior vertices, which contradicts that $C$ is $M$-alternating.
Thus $$C=v_{2n-1,2m-2r}v_{2n-1,2m-2r+1}v_{2n-2,2m-2r+1} v_{2n-2,2m-2r}v_{2n-1,2m-2r}.$$

If $v_{2n-2,2m-2r}v_{2n-2,2m-2r+1}\notin M$, then $C$ is not $M$-alternating. Hence none of cycles in ``marked subgraph'' is $M$-alternating. So we assume that $v_{2n-2,2m-2r}v_{2n-2,2m-2r+1}\in M$. Translating marked vertices right by two columns, by a similar argument, we suffice to consider the case that $v_{2n-2,2m-2r+2}v_{2n-2,2m-2r+3}\in M$. Proceeding like this, it suffices to consider the case that $M$ has the same matching form on the last $2n$ rows, i.e., $\{v_{i,2j}v_{i,2j+1}|j\in Z_m\}\subseteq M$ for $0\leq i\leq 2n-1$. Since the torsion is $2r-1$, $M$ has different matching form on the first two rows. By the previous argument, we have done.
\end{proof}

\section{\normalsize  Discussion of the maximum forcing number of $T(2n+1,2m,r)$ for $1\leq r\leq 2m$}

By Theorems \ref{odd} and \ref{even}, we obtain the maximum forcing number of $T(2n+1,2m,2r)$ for $1\leq r\leq m$.
\begin{thm}\label{mqps0} For $n\geq 1$, $m\geq 2$ and $1\leq r\leq m$, we have
\begin{equation*}
 F(T(2n+1,2m,2r))=
 \begin{cases}
 \frac{m(2n+1)+(r,m)}{2}, & \quad {if\  \frac{m}{(r,m)}\  is\  odd};\\
 \frac{m(2n+1)}{2},&\quad {otherwise}.
 \end{cases}
 \end{equation*}
\end{thm}
\begin{proof}By Lemma \ref{drawing}, $T(2n+1,2m,2r)$ has another representation $$T^*(2n+1,2m,2r)=T(2(r,m),\frac{m(2n+1)}{(r,m)},(2n+1)(\frac{m}{(r,m)}-k))$$ where $0\leq k\leq \frac{m}{(r,m)}-1$ satisfies the equation $(2r,2m)\equiv 2rk$ (mod $2m$).

If $\frac{m}{(r,m)}$ is even, then $2rk-(2r,2m)= 2mp$ for some non-negative integer $p$. That is, $rk-(r,m)= mp$. Thus $\frac{r}{(r,m)}k= \frac{m}{(r,m)}p+1$. Since $\frac{m}{(r,m)}$ is even and $\frac{m}{(r,m)}p+1$ is odd, we obtain that $k$ is an odd number. Hence $\frac{m}{(r,m)}-k$ and $(2n+1)(\frac{m}{(r,m)}-k)$ are also odd numbers.
Let $n'=(r,m)$, $m'=\frac{m(2n+1)}{2(r,m)}$ and $2r'-1=(2n+1)(\frac{m}{(r,m)}-k)$. Then $T^*(2n+1,2m,2r)=T(2n',2m',2r'-1).$
Since $0\leq k\leq \frac{m}{(r,m)}-1$, we have $2n+1\leq 2r'-1 \leq (2n+1)\frac{m}{(r,m)}=2m'$. Thus  $n+1\leq r'<m'$. By Theorem \ref{even}, we have $$F(T(2n+1,2m,2r))=F(T(2n',2m',2r'-1))=m'n'=\frac{m(2n+1)}{2}.$$

If $\frac{m}{(r,m)}$ is odd, then $2(r,m)$ is even, $\frac{m(2n+1)}{(r,m)}$ is odd.
Let $n'=(r,m)$, $2m'+1=\frac{m(2n+1)}{(r,m)}$ and $r'=(2n+1)(\frac{m}{(r,m)}-k)$. Since $0\leq k\leq \frac{m}{(r,m)}-1$, we have $2n+1\leq r'\leq (2n+1)\frac{m}{(r,m)}=2m'+1$. By Theorem \ref{odd}, we have $$F(T(2n+1,2m,2r))=F(T(2n',2m'+1,r'))=(m'+1)n'=\frac{m(2n+1)+(r,m)}{2}.$$
Now we finish the proof.
\end{proof}

For $T(2n+1,2m,2r-1)$, we have not been able to obtain a general expression for the maximum forcing number for $1\leq r\leq m$.
Therefore, computing the maximum forcing number of $T(2n+1, 2m, 2r-1)$ is an open problem.

\vspace{2mm}
\noindent{\normalsize \textbf{Acknowledgments }}  This work is supported by NSFC\,(Grant No. 12271229), start-up funds of Inner Mongolia Autonomous Region (Grant No. DC2400002165) and Inner Mongolia University of Technology (Grant No. BS2024038).
\vspace{3mm}

\noindent{\normalsize \textbf{Conflict of interest }}  The authors declare no conflict of interest.



\end{document}